\documentclass[12pt]{article}

\usepackage{fullpage}
\usepackage{amssymb}
\usepackage{amsmath}
\usepackage{amsthm}
\newtheorem{thm}{Theorem}
\newtheorem{prob}{Problem}
\newtheorem{conj}{Conjecture}
\newtheorem{cor}{Corollary}

\newtheorem{lem}{Lemma}
\newtheorem{rem}{Remark}
\begin{document}

\title{On a  problem of Sierpi\'nski}

\author{Jin-Hui Fang\\
 Department of Mathematics, Nanjing University \\ of
Information Science \& Technology,  Nanjing 210044, P. R.
CHINA\\ E-mail:fangjinhui1114@163.com \\ \and  Yong-Gao Chen\footnote{Corresponding author} \\
School of Mathematical Sciences and Institute of Mathematics,\\
Nanjing Normal University,
Nanjing 210046, P. R. CHINA\\
E-mail: ygchen@njnu.edu.cn}

\date{}

\maketitle


\renewcommand{\thefootnote}{}


\footnote{2010 \emph{Mathematics Subject Classification}:
11A41,11A67.}

\footnote{\emph{Key words and phrases}: Sierpi\'nski's problem;
consecutive primes; pairwise relatively prime.}

\renewcommand{\thefootnote}{\arabic{footnote}}
\setcounter{footnote}{0}


\begin{abstract} For any integer $s\ge 2$, let $\mu_s$ be the least integer so
that every integer $\ell >\mu_s$ is the sum of exactly $s$
integers $>1 $ which are pairwise relatively prime.  In this paper
we solve an old problem of Sierpi\'nski by determining all
$\mu_s$. As a corollary, we show that $p_2+p_3+\cdots
+p_{s+1}-2\le \mu_s\le p_2+p_3+\cdots +p_{s+1}+1100$ and the set
of integers $s\ge 2$ with $\mu_s= p_2+p_3+\cdots +p_{s+1}+1100$
has the asymptotic density 1, where $p_i$ is the $i$-th prime.
\end{abstract}

\section{ Introduction}

Let $s\ge 2$ be an integer. Denote by $\mu_s$ the least integer so
that every integer $\ell >\mu_s$ is the sum of exactly $s$
integers $>1 $ which are pairwise relatively prime. In 1964,
Sierpi\'nski \cite{Sier} asked a determination  of $\mu_s$. Let
$p_1=2$, $p_2=3, \ldots$ be the sequence of consecutive primes. In
1965, P. Erd\H os \cite{Erdos}  proved that there exists an
absolute constant $C$ with $\mu_s\le p_2+p_3+\cdots +p_{s+1}+C.$
It is easy to see that $p_2+p_3+\cdots +p_{s+1}-2$ is not the sum
of exactly $s$ integers $>1 $ which are pairwise relatively prime.
So $\mu_s\ge p_2+p_3+\cdots +p_{s+1}-2$. Let $\mu_s =
p_2+p_3+\cdots +p_{s+1}+c_s$. Then $-2\le c_s\le C$. It is easy to
see that $c_2=-2$.

Let $U$ be the set of integers of the form
$$p_2^{k_2}+p_3^{k_3}+\cdots +p_{11}^{k_{11}}-p_2-p_3-\cdots
-p_{11}\le 1100,$$ where $k_i(2\le i\le 11)$ are positive integers.
$U$ can be given explicitly by Mathematica (one may refer to the
Appendix). Let $V_s$ be the set of integers of the form
$$p_{i_1}+\cdots +p_{i_l}-p_{j_1}-\cdots -p_{j_l}\le 1100,$$ where $
2\le j_1<\cdots <j_l\le s+1<i_1<\cdots <i_l$. It is clear that $0\in
U$ and $0\in V_s$ ($l=0$). Define $U+V_s=\{ u+v\mid u\in U, \, v\in
V_s\} $. Then $U+V_s$ is finite.

In this paper the following results are proved. The main results
have been announced at ICM2010.

\begin{thm}\label{thm1}  Let $s\ge 2$ be any given positive
integer. Then
$$c_s=\max \{ 2n \mid 2n\le \min \{ 1100, p_{s+2}\},\,  n\in \mathbb{Z}, 2n\notin U+V_s\} . $$
\end{thm}

 \begin{rem}  As examples, by Theorem \ref{thm1} we have $c_{500}=16$, $c_{900}=14$, $c_{1000}=8$,
$c_{2000}=22$ (see the last section).\end{rem}

\begin{cor}\label{cor2} If $p_{s+2}-p_{s+1}>1100$, then
$$\mu_s= \sum_{i=2}^{s+1} p_i+1100.$$
In particular, the set of integers $s\ge 2$ with
$$\mu_s= \sum_{i=2}^{s+1} p_i+1100$$
has the asymptotic density 1.\end{cor}

We pose a problem here.

\begin{prob}\label{prob1} Find the least positive integer $s$
 with $\mu_s=\sum\limits_{i=2}^{s+1} {p_i} +1100$.\end{prob}

Basing on the proof of Theorem \ref{tau'} in Section \ref{sec2}, we
pose the following conjecture.

\begin{conj} For $s\ge 3$, every integer $l>p_2+p_3+\cdots
+p_{s+2}$ is the sum of exactly $s$ distinct primes.\end{conj}

This conjecture would follow from the following `` Every odd
integer $n\ge p_{s-1}+p_s+p_{s+1}+p_{s+2}$ can be written as the
sum of three prime numbers $q_1<q_2<q_3$ with $q_1\ge p_{s-1}$".
Since $p_{s-1}<n/4$, by well-known results on the odd Goldbach
problem with almost equal primes, this statement is true for all
sufficiently large $s$. Hence, Conjecture 1 is true for all
sufficiently large $s$.

Now we give a sketch proof of Theorem \ref{thm1}. For the details,
see Section \ref{sec2}.

(1) Find a ``long" interval $[1102, 3858]$ such that each even
number in this interval can be represented as
$\sum\limits_{i=2}^\infty (p_i^{t_i}-p_i)$. For any even number
$2m>3858$, there exists a prime $p_u$ such that $p_u^2-p_u\le
2m-1102<p_{u+1}^2-p_{u+1}$. Then we use the induction hypothesis
on $2m-(p_u^2-p_u)$. By these arguments we know that every even
number $n\ge 1102$ can be represented as $\sum\limits_{i=2}^\infty
(p_i^{t_i}-p_i)$, where $t_i$ are positive integers. One can
verify that 1100 cannot be represented as
$\sum\limits_{i=2}^\infty (p_i^{t_i}-p_i)$, where $t_i$ are
positive integers.

(2) Denote by $\mu_s'$ the least integer, which has the same
parity as $s$, so that every integer $\ell >\mu_s'$, which has the
same parity as $s$, can be expressed as the sum of $s$ distinct
integers $>1$ which are pairwise relatively prime. Let
$\mu_s'=p_2+\cdots +p_{s+1}+\tau_s'$. Then $\tau_s'$ is even.

For $2n>\min \{ 1100, p_{s+2}\}$, if $\min \{ 1100, p_{s+2}\} <2n\le
1100$, then $s\le 182$. By calculation we know that
$\sum\limits_{i=2}^{s+1}p_i+2n$ can be expressed as the sum of $s$
distinct odd primes. Now we assume that $2n>1100$.  If $2n$ is
``large", then we can choose a ``large" prime $q$ such that
$p_{s+2}+2n-q >\tau_s'$. By the induction hypothesis, $p_2+\cdots
+p_{s+1}+(p_{s+2}+2n-q) $ can be expressed as the sum of $s$
distinct integers $>1$, which are pairwise relatively prime. Thus
$p_2+\cdots +p_{s+1}+p_{s+2}+2n $ can be expressed as the sum of
$s+1$ distinct integers $>1$ which are pairwise relatively prime; if
$2n$ is ``small", then by (1) (we take some $t_i=1$)
$$2n=\sum\limits_{i=2}^{s+2} (p_i^{t_i}-p_i).$$
Thus
$$p_2+\cdots +p_{s+1}+p_{s+2}+2n=\sum\limits_{i=2}^{s+2}
p_i^{t_i}.$$ We can easily convert the case $p_2+\cdots
+p_{s+1}+p_{s+2}+2n+1 $ into $p_1+p_2+\cdots
+p_{s+1}+(p_{s+2}+2n-1)$ and use the induction hypothesis.

Recall that $\mu_s'$ is the least integer, which has the same
parity as $s$, so that every integer $\ell >\mu_s'$, which has the
same parity as $s$, can be expressed as the sum of $s$ distinct
integers $>1$ which are pairwise relatively prime, and
$\tau_s'=\mu_s'-(p_2+\cdots +p_{s+1})$ is even. The following
Theorem \ref{tau'} is a step in the proof of Theorem \ref{thm1},
and not an independent result.

\begin{thm}\label{tau'}
$$\tau_s'=\max \{ 2n \mid 2n\le \min \{ 1100, p_{s+2}\},\,  n\in \mathbb{Z}, 2n\notin U+V_s\} . $$
\end{thm}

\section{Preliminary Lemmas}

In this paper, $p, q_i$ are all primes. First we introduce the
following lemmas.

\begin{lem}\label{lem1}\cite[Lemma 4]{Chen}  For $x>24$ there exists a prime in $(x,
\sqrt{\frac{3}{2}}x]$.\end{lem}

\vskip 3mm

\begin{lem}\label{lem2} Every even number  $n\ge 1102$ can be
represented as $\sum\limits_{i=2}^\infty (p_i^{t_i}-p_i)$, where
$t_i$ are positive integers. The integer 1100 cannot be
represented as $\sum\limits_{i=2}^\infty (p_i^{t_i}-p_i)$, where
$t_i$ are positive integers.\end{lem}

\begin{proof} The proof is by induction on even numbers $n$.
For any sets $X,Y$ of integers, define $X+Y=\{ x+y : x\in X, y\in
Y\} $.  Let
\begin{eqnarray*}
U_4&=&\{0,3^2-3,3^3-3,3^4-3,3^5-3,3^6-3,3^7-3\}\\
&&+\{0,5^2-5,5^3-5,5^4-5\}+\{0,7^2-7,7^3-7\},\\
U_i&=&U_{i-1}\cup (U_{i-1}+\{ p_{i}^2-p_{i}\}) , \quad i=5, 6,
\cdots.\end{eqnarray*}

 By
Mathematica, we can produce each $U_i$ and verify that $[1102,
3858]\cap 2\mathbb{Z}\subseteq U_{12}$ and $1100\notin U_{12}$.

Thus, if $n$ is an even number  with $1102\le n\le 3858$, then $n$
can be represented as $\sum\limits_{i=2}^\infty (p_i^{t_i}-p_i)$,
where $t_i$ are positive integers.

Now assume that for any even number $n$ with $1102\le n<2m$
$(2m>3858)$, $n$ can be represented as $\sum\limits_{i=2}^\infty
(p_i^{t_i}-p_i)$, where $t_i$ are positive integers.

Since $2m-1102>3858-1102=53^2-53$, there exists a prime $p_u\ge 53$
with
\begin{equation}\label{x} p_u^2-p_u\le
2m-1102<p_{u+1}^2-p_{u+1}.\end{equation}
 Then
$$1102\le 2m-(p_u^2-p_u)<2m.$$ By the induction hypothesis, we have
$$2m-(p_u^2-p_u)=\sum_{i=2}^\infty
(p_i^{t_i}-p_i),$$ where $t_i$ are positive integers. Hence
\begin{equation}\label{xy}2m=\sum_{i=2}^\infty
(p_i^{t_i}-p_i)+(p_u^2-p_u).\end{equation} Now we  prove that
$t_u=1$. If this is not true, then $t_u\ge 2$ and $2m\ge
2(p_u^2-p_u)$. By (\ref{x}) we have
$$2(p_u^2-p_u)-1102\le 2m-1102<p_{u+1}^2-p_{u+1}<p_{u+1}^2-p_{u}.$$
Thus
$$2p_u^2-p_u-1102<p_{u+1}^2.$$
By $p_u\ge 53$ and Lemma \ref{lem1} we have $p_{u+1}\in (p_u,
\sqrt{\frac{3}{2}}p_u]$. Since
$$\sqrt{\frac{3}{2}}p_u\le \sqrt{2p_u^2-p_u-1102},$$ we have
$$p_{u+1}^2\le 2p_u^2-p_u-1102,$$ a contradiction. So $t_u=1$. By (\ref{xy}), we
have $2m$ can be represented as  $\sum\limits_{i=2}^\infty
(p_i^{t_i'}-p_i)$, where $t_i'$ are positive integers. Therefore,
every even number $n\ge 1102$ can be expressed as the form
$\sum\limits_{i=2}^\infty (p_i^{t_i}-p_i)$, where $t_i$ are positive
integers.

Suppose that 1100 can be expressed as $\sum\limits_{i=2}^\infty
(p_i^{t_i}-p_i)$, where $t_i$ are positive integers. Then
$p_i^{t_i}-p_i\le 1100$ for all $i$. If $t_i\ge 2$, then
$p_i^2-p_i\le 1100$. Thus $p_i<37$. So $i<12$. If $t_i\ge 3$, then
$p_i^3-p_i\le 1100$. Thus $p_i\le 7=p_4$. By $p_2^{t_2}-p_2\le
1100$ we have $t_2\le 6$. By $p_3^{t_3}-p_3\le 1100$ we have
$t_3\le 4$. By $p_4^{t_4}-p_4\le 1100$ we have $t_4\le 3$.
 Hence $1100\in U_{12}$, a contradiction. Therefore 1100 cannot  be expressed as
$\sum\limits_{i=2}^\infty (p_i^{t_i}-p_i)$, where $t_i$ are positive
integers. This completes the proof of Lemma \ref{lem2}.\end{proof}

\begin{lem}\label{lem5}  If   $2n<p_{s+2}$
and $\sum\limits_{i=2}^{s+1}p_i+2n$ is the sum of exactly $s$
integers $>1 $ which are pairwise relatively prime,
 then $\sum\limits_{i=2}^{s+1}p_i+2n$ can be expressed
as the sum of  powers of $s$ distinct primes.\end{lem}

\begin{proof} Let $$\sum_{i=2}^{s+1}p_i+2n=\sum_{i=1}^s m_i,$$
where $ 1<m_1<\cdots<m_s$ and $ (m_i, m_j)=1$ for $1\le i, j\le
s$, $i\neq j.$ By comparing the parities we know that these $s$
integers $m_i$ must all be odd. If one of these $s$ integers has
at least two distinct prime factors, then the sum of these $s$
integers is at least $3\times 5+p_4+\cdots
+p_{s+2}=p_2+\cdots+p_{s+1}+p_{s+2}+7$. This contradicts $2n\le
p_{s+2}$.   This completes the proof of Lemma
\ref{lem5}.\end{proof}

\section{Proof of Theorem \ref{tau'}}

 For $s\ge 2$, let
\begin{eqnarray*}H(s)&=&\{ p_j-p_i : 2\le i\le s+1<j\le 185\}\\
&&\bigcup \{ p_u+p_v-p_s-p_{s+1} : s\le u\le 105, u<v\le 180\}
.\end{eqnarray*} By Mathematica, for $2\le s \le 182$ we find that
$[p_{s+2}, 1100]\cap 2\mathbb{Z} \subseteq H(s)$. Thus, for
$p_{s+2}< 2n\le 1100$, $\sum\limits_{i=2}^{s+1}p_i+2n$ can be
expressed as the sum of $s$ distinct odd primes.

Let $h_s $ be the largest even number $2n\le 1100$ such that
$\sum\limits_{i=2}^{s+1}p_i+2n$ cannot be expressed as the sum of
$s$ distinct integers $>1 $ which are pairwise relatively prime.
Noting that $p_{s+2}>1100$ for  $s\ge 183$, by the above arguments
we have $h_s\le \min\{ 1100,p_{s+2}\}$ for all $s\ge 2$.

 We will use induction on $s$ to prove that $\tau_s'=h_s $ for all $s\ge
2$.

For every even number $\ell >6$, we have $\phi(\ell )>2$, where
$\phi(\ell )$ is the Euler's totient function. Hence there exists
an integer $n$ with $2\le n\le \ell -2$ and $(n,\ell )=1$. So
$$\ell =n+(\ell -n),\quad (n, \ell -n)=1,\quad n\ge 2, \ell -n\ge 2.$$ Thus
$\tau_2'=-2=h_2$. Suppose that
 $\tau_s'=h_s$. Now we prove that
$\tau_{s+1}'=h_{s+1}$.

 Let $\ell $ be an integer which has the same parity as $s+1$. Write
 $$\ell =\sum\limits_{i=2}^{s+2}p_i+2n.$$
 Then $2n$ is an even number. By the definition of $\tau_{s+1}'$ and
 $h_{s+1}$, it is enough to prove that if $2n>1100$, then $\sum\limits_{i=2}^{s+2}p_i+2n$
 can be expressed as the sum of
$s+1$ distinct integers $>1 $ which are pairwise relatively prime.

Assume that $2n>1100$. Write $2t=2n-\tau_s'$.  By $\tau_s'=h_s\le
p_{s+2}$ we have $p_{s+2}+2t=p_{s+2}+2n-\tau_s'\ge 2n>1100$. By
Lemma \ref{lem1} there exists an odd prime $q$ with $\frac{2}{3}
(p_{s+2}+2t)<q< p_{s+2}+2t$. Then
$$\ell -q>\ell -p_{s+2}-2t=\sum_{i=2}^{s+1}p_i+\tau_s'.$$
 Since $$\ell -q \equiv s \pmod 2,$$ by the induction hypothesis, we have
$$\ell -q=n_1+\cdots+n_s,$$ where $1<n_1<\cdots<n_s$ and $(n_i, n_j)=1$ for
$1\le i, j\le s$, $i\neq j$ . By $\ell -q\equiv s \pmod 2$ and
$(n_i, n_j)=1$ for $1\le i,j\le s$, $i\not= j$, we have $2 \nmid
n_i$ for $1\le i\le s$.

If $q>n_s$, we are done. Now we assume that $q\le n_s$. By $\ell
-q=n_1+\cdots+n_s,$ we have \begin{equation}\label{kl} \ell \ge
2q+p_2+\cdots+p_s>
\frac{4}{3}p_{s+2}+\frac{8}{3}t+p_2+\cdots+p_s.\end{equation} By
(\ref{kl}) and
$$\ell =\sum\limits_{i=2}^{s+2}p_i+2t+\tau_s',$$ we have
\begin{equation}\label{kw}\frac{1}{3}p_{s+2}-p_{s+1}+\frac{2}{3}t<\tau_s'.\end{equation}
Noting that $\tau_s'\le p_{s+2}$, by (\ref{kw}) we have
\begin{equation}\label{kx}2n=2t+\tau_s'< 4\tau_s'+3p_{s+1}-p_{s+2}<6p_{s+2}.\end{equation}

Since $2n>1100$, by Lemma \ref{lem2} we have
\begin{equation}\label{ky}2n=\sum_{i=2}^\infty (p^{t_i}_{i}-p_{i}),
\quad t_i\ge 1, \, i=2,3,\ldots .\end{equation} For $i\ge s+3$, by
(\ref{kx}) and (\ref{ky}) we have
$$p_{s+3}^{t_i}-p_{s+3}\le p_i^{t_i}-p_i\le 2n<6p_{s+2}.$$
Since $p_{s+3}-1\ge p_5-1=10$, we have $t_i=1$ for all $i\ge s+3$.
Hence
$$\ell =\sum\limits_{i=2}^{s+2}p_i+2n=\sum_{i=2}^{s+2}p_i
+\sum_{i=2}^{s+2} (p^{t_i}_{i}-p_{i})=\sum_{i=2}^{s+2} p^{t_i}_{i}.$$
Thus we have proved that if $\ell =\sum\limits_{i=2}^{s+2}p_i+2n$
cannot be expressed as the sum of $s+1$ distinct integers $>1 $
which are pairwise relatively prime, then $2n\le 1100$. By the
definition of $h_{s+1}$ and $\tau_{s+1}'$, we have
$\tau_{s+1}'=h_{s+1}$. Therefore, $\tau_s'=h_s $ for all $s\ge 2$.

Now we have proved that $\tau_s'=h_s$ is the largest even number
$2n\le 1100$ such that $\sum\limits_{i=2}^{s+1}p_i+2n$ cannot be
expressed as the sum of $s$ distinct integers $>1 $ which are
pairwise relatively prime and $\tau_s'=h_s\le \min\{
1100,p_{s+2}\}$.

In order to prove Theorem \ref{tau'}, it is enough to prove that
$\tau_s'\notin U+V_s$ and   if $2n$ is an even number with $\tau_s'
<2n\le \min\{ 1100, p_{s+2}\}$, then $2n\in U+V_s$.

Let $2n$ be an even number with $\tau_s' <2n\le \min\{ 1100,
p_{s+2}\}$. Now we prove that $2n\in U+V_s$. By Lemma \ref{lem5} and
the definition of $\tau_s'$, we have
$$p_2+\cdots +p_{s+1}+2n=p_{l_1}^{\alpha_1}+\cdots +p_{l_s}^{\alpha_s},$$
 where $2\le l_1<\cdots <l_s$ and $\alpha_i\ge
1(1\le i\le s)$. If $l_1\ge s+2$, then $l_i\ge s+1+i (1\le i\le s)$.
Thus $l_s\ge 2s+1\ge 5$ and $p_{l_s}\ge p_5=11$. Hence
\begin{eqnarray*}
2n&=&p_{l_1}^{\alpha_1}+\cdots +p_{l_s}^{\alpha_s}-(p_2+\cdots +p_{s+1})\\
&\ge & p_{s+2}+\cdots +p_{2s+1}-(p_2+\cdots +p_{s+1})\\
&\ge & p_{s+2}+\cdots +p_{2s}+11-(p_2+\cdots +p_{s+1})\\
&> & p_{s+2},
\end{eqnarray*}
a contradiction with $2n\le \min\{ 1100, p_{s+2}\}$. So $l_1\le
s+1$. Let $r$ be the largest index with $l_r\le s+1$. If $r=s$, then
$l_i= i+1 (1\le i\le s)$. Thus
\begin{equation}\label{uv}2n=(p_{2}^{\alpha_1}-p_2)+\cdots
+(p_{s+1}^{\alpha_s}-p_{s+1}).\end{equation} If $r<s$, let
$$\{ 2, 3,\ldots , s+1\} =\{ l_1,\ldots , l_r\} \bigcup \{
j_1,\ldots , j_{s-r}\} $$ with $j_1<\cdots <j_{s-r}$. Hence
\begin{equation}\label{vu}2n=(p_{l_1}^{\alpha_1}-p_{l_1})+\cdots +(p_{l_r}^{\alpha_r}-p_{l_r})
+p_{l_{r+1}}^{\alpha_{r+1}}+\cdots
+p_{l_s}^{\alpha_s}-p_{j_1}-\cdots -p_{j_{s-r}}.\end{equation} For
$1\le i\le r$, if $\alpha_i\ge 2$, then by (\ref{uv}) and (\ref{vu})
we have
$$p_{l_i}(p_{l_i}-1)\le 2n\le 1100.$$
Thus $p_{l_i}\le 31$ and $l_i\le 11$. Hence
\begin{equation}(p_{l_1}^{\alpha_1}-p_{l_1})+\cdots +(p_{l_r}^{\alpha_r}-p_{l_r})\in U.\end{equation}
For $r<i\le s$, if $\alpha_i\ge 2$, then
\begin{eqnarray*}&&p_{l_{r+1}}^{\alpha_{r+1}}+\cdots +p_{l_s}^{\alpha_s}-p_{j_1}-\cdots -p_{j_{s-r}}\\
&\ge & p_{s+2}^2+(s-r-1)p_{s+3}-(s-r)p_{s+1}>p_{s+2}\ge
2n,\end{eqnarray*} a contradiction. So $\alpha_i=1$ for all $r<i\le
s$. By (\ref{vu}) we have $$p_{l_{r+1}}^{\alpha_{r+1}}+\cdots
+p_{l_s}^{\alpha_s}-p_{j_1}-\cdots -p_{j_{s-r}}\le 2n\le 1100.$$
Hence
\begin{equation}\label{u}
p_{l_{r+1}}^{\alpha_{r+1}}+\cdots +p_{l_s}^{\alpha_s}-p_{j_1}-\cdots
-p_{j_{s-r}}=p_{l_{r+1}}+\cdots +p_{l_s}-p_{j_1}-\cdots
-p_{j_{s-r}}\in V_s.\end{equation} By (\ref{uv}) - (\ref{u}) we have
$2n\in U+V_s$.

In order to prove Theorem \ref{tau'}, it suffices to prove that
$\tau_s'\notin U+V_s$.

Suppose that $\tau_s'\in U+V_s$. Then
$$\tau_s'=\sum_{i=2}^{11} (p_i^{\beta_i}-p_i)+p_{i_1}+\cdots +p_{i_l}-p_{w_1}-\cdots
-p_{w_l},$$ where $\beta_i(2\le i\le 11)$ are positive integers and
$w_1<\cdots <w_l\le s+1<i_1<\cdots <i_l$. Let
$$\sum_{i=2}^{11}
(p_i^{\beta_i}-p_i)=\sum_{i=1}^m(p_{e_i}^{d_{i}}-p_{e_i}),$$ where
$2\le e_1<\cdots <e_m\le 11$ and $d_i\ge 2(1\le i\le m)$. Since
$$p_{e_m}(p_{e_m}-1)\le p_{e_m}^{d_{m}}-p_{e_m}\le \tau_s'\le
p_{s+2},$$ we have $e_m\le s+1$. If $w_1\le e_m$, then
\begin{eqnarray*}
\tau_s'&=&\sum_{i=1}^m(p_{e_i}^{d_{i}}-p_{e_i})+p_{i_1}+\cdots
+p_{i_l}-p_{w_1}-\cdots -p_{w_l}\\
&\ge & p_{e_m}^{d_{m}}-p_{e_m}-p_{w_1}+p_{s+2}\\
&\ge & p_{e_m}(p_{e_m}-2)+p_{s+2}>p_{s+2},\end{eqnarray*} a
contradiction with $\tau_s'\le \min\{ 1100, p_{s+2}\}$. Hence
$e_m<w_1$. Thus $$2\le e_1<\cdots <e_m<w_1<\cdots < w_l\le
s+1<i_1<\cdots <i_l.$$ Let
$$\{ f_1,\ldots , f_{s-m-l}\} =\{ 2,\ldots , s+1\} \setminus \{
e_1, \ldots , e_m, w_1, \ldots , w_l\} .$$ Then
$$p_2+\cdots +p_{s+1}+\tau_s' =\sum_{i=1}^m
p_{e_i}^{d_{i}}+p_{f_1}+\cdots +p_{f_{s-m-l}}+p_{i_1}+\cdots
+p_{i_l}.$$ Since $e_1,\ldots , e_m, f_1, \ldots , f_{s-m-l},
i_1,\ldots , i_l$ are pairwise distinct, this contradicts the
definition of $\tau_s'$. Hence $\tau_s'\notin U+V_s$. This
completes the proof of Theorem \ref{tau'}.

\section{Proofs of Theorem \ref{thm1} and  Corollary \ref{cor2}}\label{sec2}

It is easy to see that $c_2=-2$ and $\{ 0, 2, 4, 6\} \in V_2$. Thus,
by $0\in U$, all even numbers $2n$ with $-2<2n\le \min\{ 1100,
p_{2+2}\}$ are in $U+V_2$. So Theorem \ref{thm1} is true for $s=2$.

Now we assume that $s>2$.

In order to prove Theorem \ref{thm1}, by Theorem \ref{tau'} it is
enough to prove that for any odd number $2k+1>\tau_s'$, $p_2+\cdots
+p_{s+1}+2k+1$ can be expressed as the sum of $s$ distinct integers
$>1 $ which are pairwise relatively prime. Since $\tau_s'\ge -2$, we
have $k\ge -1$. If $k=-1$, then
$$p_2+\cdots+p_{s+1}+2k+1=p_1+p_3+p_4+\cdots +p_{s+1}.$$
If $k=0$, then
$$p_2+\cdots+p_{s+1}+2k+1=p_1^2+p_3+p_4+\cdots +p_{s+1}.$$
Now we assume that $k\ge 1$. By Theorem \ref{tau'} we have
$p_{s+1}+2k-1> \tau_{s-1}'$. Hence
$$p_2+\cdots +p_s+(p_{s+1}+2k-1)=n_1+\cdots +n_{s-1},$$
where $1<n_1<\cdots<n_{s-1}$ and $(n_i, n_j)=1$ for $1\le i, j\le
s-1$, $i\neq j$. By $p_2+\cdots +p_s+(p_{s+1}+2k-1)\equiv s-1 \pmod
2$ and $(n_i, n_j)=1$ for $1\le i,j\le s-1$, $i\not= j$, we have $2
\nmid n_i$ for $1\le i\le s-1$. Thus
$$p_2+\cdots +p_s+(p_{s+1}+2k+1)=2+n_1+\cdots +n_{s-1}$$
is the required form.

This completes the proof of Theorem \ref{thm1}.

\begin{proof}[Proof of  Corollary \ref{cor2}.]
 Suppose that
$p_{s+2}-p_{s+1}>1100$. Then $V_s=\{ 0\} $. Since $1100\notin U$,
we have $1100\notin U+V_s$. By Theorem \ref{thm1} we have
$c_s=1100$. This completes the proof of the first part of
Corollary \ref{cor2}.

The second part now follows from the fact that the number of
primes $p\le x$, such that $p+k$ is prime, is bounded above by
$c\frac{x}{\log^2 x}$, where $c$ depends only on $k$ ( Brun
\cite{Brun}, S\'andor, Mitrinovi\'c and Crstici
\cite[p.238]{Sandor}, Wang \cite{Wang}). This completes the proof
of the second part of Corollary \ref{cor2}.\end{proof}

\section{Final Remarks}

 Let $A=([2,1100]\cap 2\mathbf{N})\setminus U$ and for $t<s$, let
\begin{eqnarray*}V_s(t)&=& \{ p_{s+2+i}-p_{s+1-j} \mid 0\le i,j\le
t\}\\
 &&\cup \{ p_{s+2+i}+p_{s+2+j}-p_{s+1-u}-p_{s+1-v} \mid 0\le i<j\le t, 0\le u<v\le
 t\} .\end{eqnarray*}
Let
$$a(s,t)= \max (A\setminus (U+V_s(t))).$$
If
$$a(s,t)<\min\{ p_{s+2+t}-p_{s+1}, p_{s+2}-p_{s+1-t},
p_{s+3}+p_{s+2}-p_{s+1}-p_s\},$$ then
$$a(s,t)= \max (A\setminus (U+V_s)).$$

Noting that $A=([2,1100]\cap 2\mathbf{N})\setminus U$, by Theorem
\ref{thm1} we have $c_s=a(s,t)$. Taking $t=5$, by Mathematica we
find that $c_{500}=16$, $c_{900}=14$, $c_{1000}=8$, $c_{2000}=22$,
etc.

\subsection*{Acknowledgements} The authors are supported
by the National Natural Science Foundation of China, Grant
No.11126302, 11071121. The first author is also supported by the
Natural Science Foundation of the Jiangsu Higher Education
Institutions, Grant No. 11KJB110006 and the Foundation of Nanjing
University of Information Science \& Technology. We would like to
thank the referee for his/her comments.

\newpage
\centerline{\bf Appendix}

\vskip 3mm

$U=\{$ 0, 6, 20, 24, 26, 42, 44, 48, 62, 66, 68, 78, 86, 98, 110,
116, 120, 126, 130, 134, 136, 140, 144, 152, 154, 156, 158, 162,
168, 172, 176,  178, 180, 182, 186, 188, 196, 198, 200, 204, 208,
218, 222, 224, 230,  234, 236, 240, 242, 250, 254, 260, 266, 272,
276, 278, 282, 286, 290,  292, 296, 298, 300, 302, 308, 310, 314,
316, 318, 320, 324, 328, 332,  334, 336, 338, 340, 342, 344, 348,
350, 352, 354, 356, 358, 360, 362,  364, 366, 368, 370, 380, 382,
384, 386, 388, 390, 392, 396, 398, 402,  404, 406, 408, 410, 412,
414, 416, 420, 424, 426, 428, 430, 434, 438,  440, 444, 446, 448,
450, 452, 454, 456, 458, 460, 462, 464, 466, 468,  470, 472, 476,
478, 480, 482, 486, 490, 492, 494, 496, 498, 500, 502,  504, 506,
508, 510, 512, 514, 516, 518, 520, 522, 524, 526, 528, 530,  532,
534, 536, 538, 540, 542, 544, 546, 548, 550, 554, 558, 560, 562,
564, 566, 568, 570, 572, 574, 576, 578, 580, 582, 584, 586, 590,
592,  596, 600, 602, 604, 606, 608, 612, 614, 616, 618, 620, 622,
624, 626,  628, 632, 634, 636, 638, 640, 642, 644, 646, 650, 652,
656, 658, 660,  662, 664, 666, 668, 670, 674, 676, 678, 680, 682,
684, 686, 688, 690,  692, 694, 696, 698, 700, 702, 704, 706, 710,
712, 714, 718, 722, 724,  726, 728, 730, 732, 734, 736, 738, 740,
742, 744, 746, 748, 750, 752,  754, 756, 758, 760, 762, 764, 766,
768, 770, 772, 776, 778, 780, 782,  784, 786, 788, 790, 792, 794,
796, 798, 800, 802, 804, 806, 808, 810,  812, 814, 816, 818, 820,
822, 824, 826, 830, 832, 834, 836, 838, 840,  842, 844, 846, 848,
850, 852, 854, 856, 858, 860, 862, 864, 866, 868,  870, 872, 874,
876, 878, 880, 882, 884, 886, 888, 890, 892, 894, 896,  898, 900,
902, 904, 906, 908, 910, 912, 914, 916, 918, 920, 922, 924,  926,
928, 930, 932, 934, 936, 938, 940, 942, 944, 946, 948, 950, 952,
954, 956, 958, 960, 962, 964, 966, 968, 970, 972, 974, 976, 978,
980,  982, 984, 986, 988, 990, 992, 994, 996, 998, 1000, 1002, 1004,
1006,  1008, 1010, 1012, 1014, 1016, 1018, 1020, 1022, 1024, 1026,
1028,  1030, 1032, 1034, 1036, 1038, 1040, 1042, 1044, 1046, 1048,
1050,  1052, 1054, 1056, 1058, 1060, 1062, 1064, 1066, 1068, 1070,
1072,  1074, 1076, 1078, 1080, 1082, 1084, 1086, 1088, 1090, 1092,
1094,  1096, 1098 $\} $.

\end{document}